\documentclass{amsart}
\usepackage[utf8]{inputenc}

\usepackage{mathptmx}
\usepackage[margin=1.25in]{geometry}
\usepackage{graphics}
\usepackage{thmtools}

\usepackage{amsthm}
\usepackage{amsbsy,amsmath,amssymb,amscd,amsfonts}
\usepackage{xstring}
\usepackage[pagebackref=true]{hyperref}
\usepackage[nameinlink,capitalize,noabbrev]{cleveref}

\DeclareRobustCommand{\hrefs}[1]{%
  \def\temphrefsurl{#1}
  \IfBeginWith{\temphrefsurl}{http://}{%
    \StrGobbleLeft{\temphrefsurl}{7}[\temphrefsdisplayurl]
  }{%
    \IfBeginWith{\temphrefsurl}{https://}{%
      \StrGobbleLeft{\temphrefsurl}{8}[\temphrefsdisplayurl]
    }{%
      \def\temphrefsdisplayurl{\temphrefsurl}
    }%
  }%
  \href{\temphrefsurl}{\texttt{\temphrefsdisplayurl}}
}

\usepackage{graphicx,float,latexsym,color}
\usepackage[export]{adjustbox}

\usepackage[font={scriptsize,it},skip=5pt]{caption}
\usepackage{subcaption}

\usepackage{makecell}

\usepackage[dvipsnames]{xcolor}

\newtheorem*{theorem*}{Theorem}
\newtheorem{observation}{Observation}
\newtheorem{proposition}{Proposition}
\newtheorem{conjecture}{Conjecture}
\newtheorem{corollary}{Corollary}
\newtheorem{lemma}{Lemma}
\theoremstyle{remark}

\theoremstyle{definition}
\newtheorem{definition}{Definition}

\hypersetup{
    pdftoolbar=true,        
    pdfmenubar=true,        
    pdffitwindow=false,     
    pdfstartview={FitH},    
    colorlinks=true,       
    linkcolor=OliveGreen,          
    citecolor=blue,        
    filecolor=black,      
    urlcolor=red           
}

\usepackage{listings}
\usepackage{lineno}

\newcommand{\tb}[1]{\textbf{#1}}

\newcommand{\ti}[1]{\textit{#1}}

\usepackage{comment}
\usepackage{microtype}
\usepackage{footnote}

\renewcommand{\O}{\mathcal{O}}

\newcommand{\E}{\mathcal{E}}
\newcommand{\K}{\mathcal{K}}

\renewcommand{\L}{\mathcal{L}}

\newcommand{\torp}[2]{\texorpdfstring{#1}{#2}}

\newcommand{\rc}{\raisebox{0.3ex}{,}}
\newcommand{\rd}{\raisebox{0.3ex}{.}}

\title{Four special Poncelet triangle families about the incircle}
\author[R. Garcia]{Ronaldo A. Garcia}
\author[M. Helman]{Mark Helman}
\author[D. Reznik]{Dan Reznik} 
\begin{document}
\maketitle
\vspace{-1.5em}  
\begin{center}
\today
\end{center}

\begin{abstract}
We describe four special families of ellipse-inscribed Poncelet triangles about the incircle which maintain certain triangle centers stationary and which also display interesting conservations.
\end{abstract}

\section{Introduction}

As special cases to \cite{garcia2024-incircle}, we introduce four special families of Poncelet triangles inscribed in an ellipse $\E$ and circumscribing their fixed incircle, let $C$ be its center, i.e., the \ti{incenter} $X_1$ (triangle center notation $X_k$ are after Kimberling \cite{etc}). Referring to \cref{fig:four-families}, the four Poncelet families are:
\begin{itemize}
\item focal-$X_1$: the caustic is centered on a focus of $\E$; \item iso-$X_2$: caustic is centered on a special point on $\E$'s minor axis such that the barycenter $X_2$ is stationary;
\item focal-$X_4$: caustic is centered on a special point $\E$'s major axis such that the orthocenter $X_4$ is stationary on a focus of $\E$;
\item iso-$X_7$: caustic is centered on a another special point on $\E$'s major axis such that the Gergonne point $X_7$ is stationary.
\end{itemize}

In the sections below, $R$, $r$, $l_i,i=1,2,3$ refer to a triangle's circumradius, inradius, and sidelengths, respectively, see \cite{mw} for definitions.

\begin{figure}
\centering
\includegraphics[width=\linewidth]{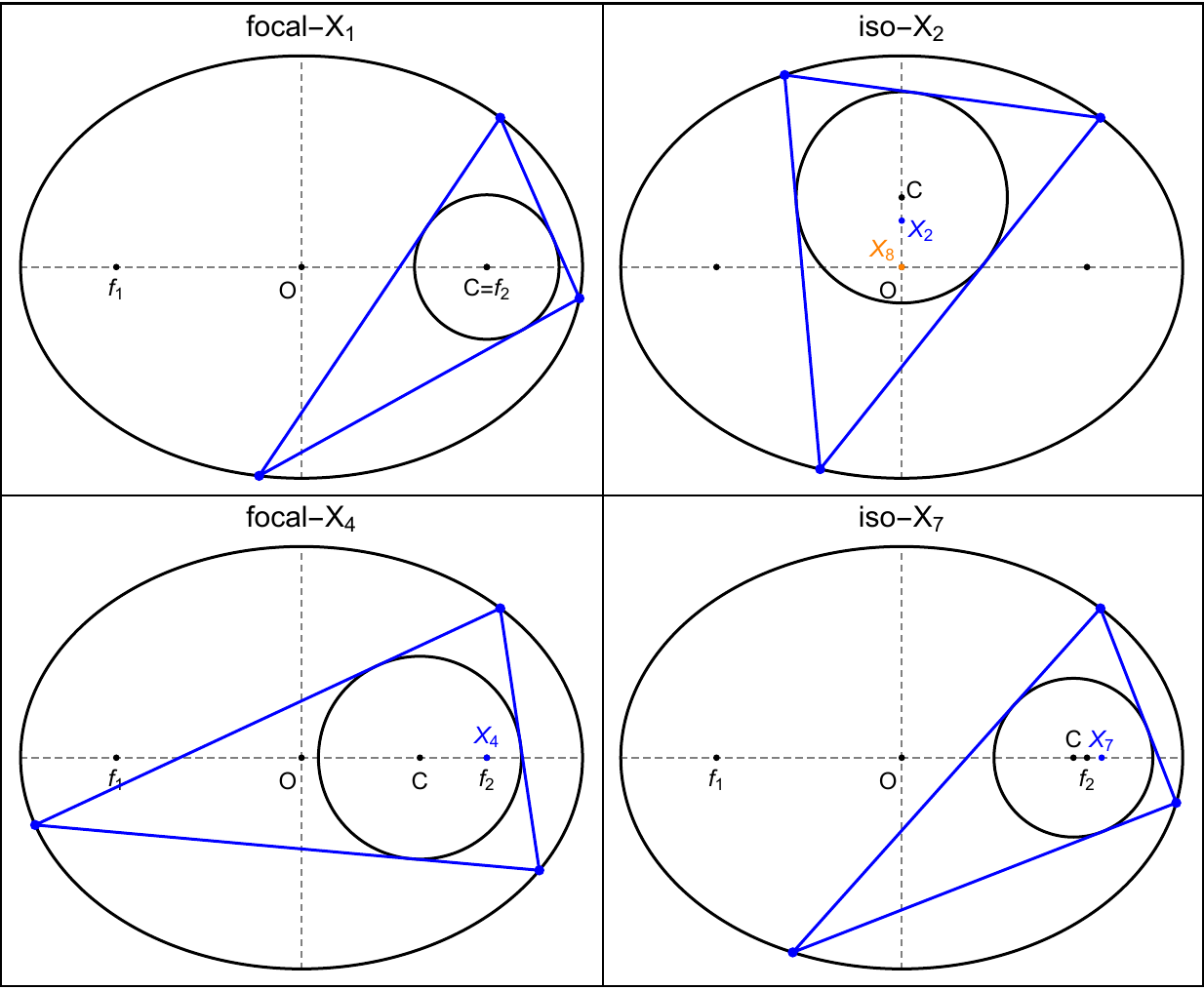}
\caption{Four families of ellipse-inscribed Poncelet triangles about the incircle. \tb{top-left}: focal-$X_1$ ($C$ at a focus of $\E$); \tb{top-right}: iso-$X_2$ ($C$ on minor axis of $\E$, barycenter $X_2$ stationary); \tb{bot.-left}: focal-$X_4$ ($X_4$ is a focus of $\E$); \tb{bot.-right}: iso-$X_7$ ($C$ on major axis of $\E$, $X_7$ stationary). Video: \hrefs{https://youtu.be/nsHDfX6\_7mA}}
\label{fig:four-families}
\end{figure}

\section{\torp{Focal-$X_1$}{Focal-X(1)}}

Let $\E$ be given by $(x/a)^2+(y/b)^2=0$. Referring to \cref{fig:four-families} (top left):

\begin{proposition}
The caustic centered at $C_1=[c,0]$ admits a Poncelet family of triangles with $r_1$ given by:
\[r_1={\frac {{b}^{2}}{{c}^{2}} \left( \sqrt {{a}^{2}+{c}^{2}}-a \right) }\cdot\]
\end{proposition}

\begin{proof}
This follows from \cite[Prop. 2]{garcia2024-incircle}:
the radius $r$ for a circular caustic (of Poncelet triangles) with center at $[x_c,y_c]$ is given by:
\[ r=\frac{b\sqrt{a^4-c^2 x_c^2}- a\sqrt{b^4+c^2 y_c^2}}{c^2}\cdot \qedhere \]
\end{proof}
\begin{observation}
The focal-$X_1$ family is the polar image of the vertices of Chapple's Porism (a porism of triangles interscribed between two circles \cite{odehnal2011-poristic}), with respect to the circumcircle, i.e., it is its tangential triangle. The former's incenter $X_1$ coincides with the latter's circumcenter $X_3$.
\end{observation}

Bicentric $n$-gons (a generalization of Chapple's porism) conserve the sum of cosines of their internal angles $\theta_i$ \cite{roitman2021-bicentric}. On the other hand, the polar family with respect to the circumcircle conserves $\sum_{i=1}^{n}{\sin(\theta_i/2)}$, for all $n\geq 3$  \cite[prop.26]{bellio2022-parabola-inscribed}.

\begin{proposition}
For the focal-$X_1$ family, the sum of half-angle sines is invariant and given by: 
\[ \sum_{i=1}^{3}\sin\frac{\theta_i}{2} = \frac{c^2 -a^2 + a \sqrt{a^2 + c^2}}{c^2} \cdot \]
\end{proposition}

\section{\torp{Iso-$X_2$}{Iso-X(2)}}

Referring to \cref{fig:four-families} (top right), it follows from the expressions for the center $C_2$ and semi-axis lengths $a_2,b_2$ for the elliptic locus of the barycenter $X_2$ \cite[Prop. 3]{garcia2024-incircle} that:

\begin{proposition}
For a circular caustic with center $C_2=\left[0,\frac{c\,b}{2 a}\right]$ and radius $r_2=\frac{b}{2}$, the barycenter $X_2$ is stationary at $\left[0,\frac{c\,b}{3a}\right]$.
\end{proposition}

\begin{corollary}
Over iso-$X_2$ triangles, Nagel's point $X_8$ is stationary at $E$'s center.
\end{corollary}

\begin{proof}
Direct from the location of the stationary barycenter and the fact that $|X_1 X_8| = 3 |X_1 X_2|$ \cite[Incenter, eqn.11]{mw}. In fact,
$|X_1 X_8|=3(c\,b/(2a)-c\,b/(3a))=c\,b/(2a)=|C_2|$.
\end{proof}

\begin{corollary}
The Spieker center $X_{10}$, incenter of the medial triangle and the midpoint of the incenter $X_1$ and Nagel's point $X_8$ \cite{etc} will be stationary on $E$'s minor axis.
\end{corollary}

\begin{observation}
Since the barycenter $X_2$ is stationary, the family conserves $|X_1 X_2|$, which for any triangle with sidelengths $l_i$ is given by \cite[Triangle Centroid, eqn.11]{mw}:
\[ |X_1-X_2|^2 = -\frac{\sum{l_i^3}+9\prod{l_i}-2 \left(l_2 l_1^2+l_3 l_1^2+l_2^2 l_1+l_3^2 l_1+l_2 l_3^2+l_2^2 l_3\right)}{9\sum{l_i}} \rd\]
\end{observation}

\begin{definition}
The MacBeath Poncelet family, shown in \cref{fig:macbeath-dual} (left), are triangles inscribed in a circle, and circumscribing the \ti{MacBeath inconic}, whose foci are the circumcenter $X_3$ and the orthocenter $X_4$, Its center is the Euler center $X_5$, the midpoint of the circumcenter $X_3$ and the orthocenter $X_4$ \cite[MacBeath inconic]{mw}.
\end{definition}

\begin{lemma}
\label{lem:x2-macbeath}
The barycenter $X_2$ of the MacBeath family is stationary and lies on the major axis of the inconic.
\end{lemma}

\begin{proof}
The foci of the caustic are the circumcenter $X_3$ and the orthocenter $X_4$, therefore the barycenter $X_2$ must be fixed, $|X_3 X_2|=|X_3 X_4|/3$ \cite[Triangle Centroid, eqn.(13)]{mw}. Consider also that MacBeath family are the intouch triangles of bicentric polygons (Chapple's porism). The the barycenter $X_2$ of the intouch triangle is the Weill point $X_{354}$ of the reference \cite{etc}, shown to be stationary over the bicentric family \cite[Thm.~4.1]{odehnal2011-poristic}. Since the barycenter $X_2$ lies on the Euler line, which also contains $X_3,X_4$, it must lie on the inconic's axis.
\end{proof}


\begin{definition}[Isogonal conjugate]
\label{def:isog}
The isogonal conjugate $P^{\dagger}$ of a point $P$ with respect to a triangle $T$ is the point of concurrence of the cevians of $P$ reflected upon the angle bisectors. If the barycentrics of $P$ are $[z_1:z_2:z_3]$,  $P^{\dagger}=[l_1^2/z_1 : l_2^2/z_2 : l_3^2/z_3]$, where the $l_i$ are the sidelengths.
\end{definition}

\begin{proposition}
\label{prop:affine-macbeath}
Given an outer ellipse $\E$ with axes $(a,b)$ centered on $O$ and a point $O_c$ in the interior of a concentric, axis-aligned ellipse with axes $(a/2, b/2)$, there is a caustic $\K$ which admits a Poncelet triangle family. Over this family, the barycenter $X_2$ is stationary, with the latter lying on $O O_c$.
\end{proposition}

\begin{proof}
Referring to \cref{fig:x2-macbeath}, let $\mathcal{A}$ be an affine transform that sends an outer Poncelet conic $\E$ centered at $O$ (left) to a circle $\E'$ (right), centered at $O'=\mathcal{A}(O)$ (conic centers are equivariant under affinities). Since $O_c$ is as specified, $O_c'=\mathcal{A}(O_c)$ will lie within a circle concentric with $\E'$ and of half its radius.

Let $\K'$ be the unique Poncelet caustic with a first focus at $O'$ (the circumcenter $X_3$ of its triangles) and the other at the reflection of $\O'$ about $O_c'=\mathcal{A}(O_c)$. Since foci of a triangle's inconic are an isogonal conjugate pair \cite{beluhov2007-isogonal-foci} (see also \cref{def:isog}), the second focus must be the fixed orthocenter $X_4$ of the family (also within $\E'$), i.e., this is the MacBeath inconic. 

Let $\K$ be the image of $\K'$ under $\mathcal{A}^{-1}$. The the barycenter $X_2$ of the family in $(\E,\K)$ will be stationary because (i) by \cref{lem:x2-macbeath}, the MacBeath's centroid is stationary, and (ii) the centroid is the only triangle center which is equivariant with respect to affinities \cite{etc}. 
The barycenter $X_2$ must lie on $O C$ because affinities preserve collinearities. Notice it will nevertheless not lie on the major axis of $\K$ (foci are not equivariant).
\end{proof}

\begin{figure}
\centering
\includegraphics[width=\linewidth]{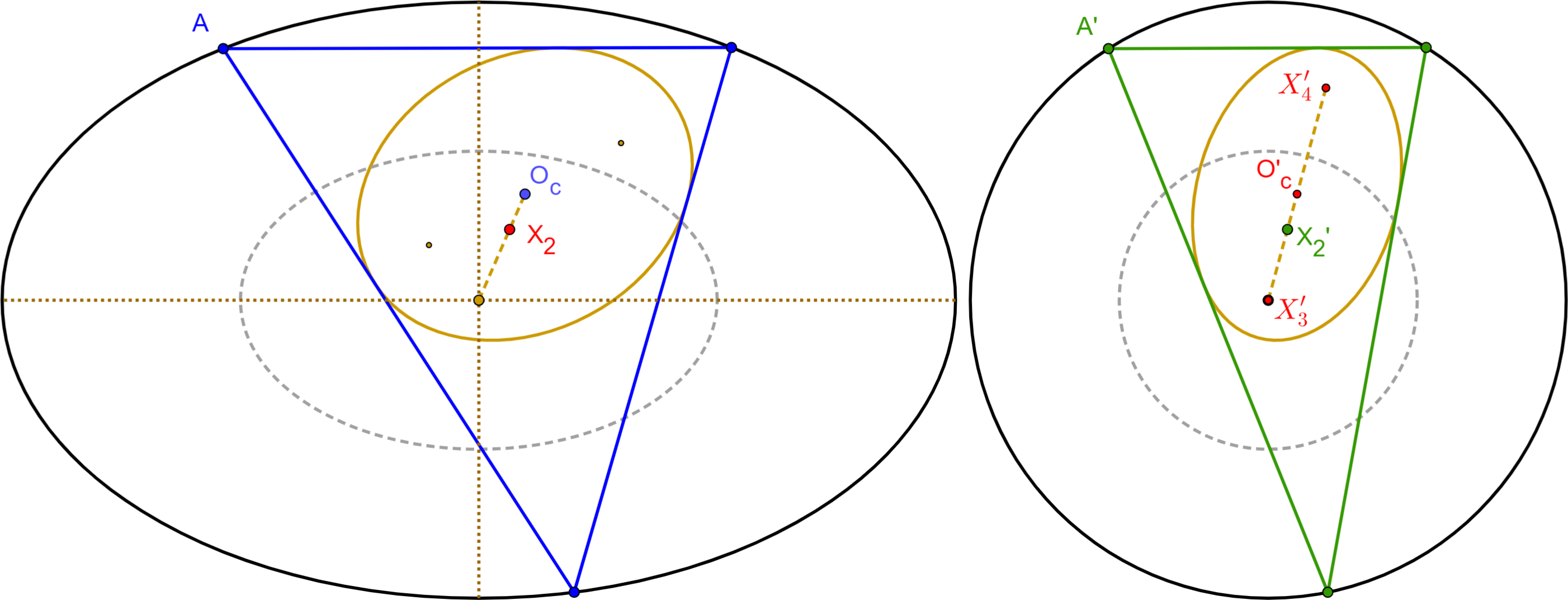}
\caption{Constructions for the proof of \cref{prop:affine-macbeath}: an affine transformation exists such that the outer ellipse (right) is sent to a circle (right). The dashed half-size ellipse in the left (half radius circle on the right) delimits the region for acceptable $O_c$ ($O_c'$ on the right). Video: \url{https://youtu.be/yDz\_XopFw2Y}}
\label{fig:x2-macbeath}
\end{figure}

Referring to \cref{fig:macbeath-n45}, experimental evidence suggests that for MacBeath-like Poncelet families of $n$-gons, $n>3$, i.e., they are circle-inscribed and the caustic has a focus at the center of said circle, both vertex and area centroids $C_0$ and $C_2$ remain stationary on the caustic axis. When $n=4$, $C_0$ lies at the caustic's center. 

\begin{corollary}
Given an ellipse $E$ centered at $O$ and a point $C$ in its interior, there is a caustic centered on $C$ for which both $c_0$ and $c_2$ are stationary on $O C$. 
\end{corollary}

Indeed, over such families, the perimeter centroid, in general not expected to sweep a conic over Poncelet \cite{sergei2016-com}, is experimentally found to sweep an ellipse with major axis identical to the caustic's. 

\begin{figure}
\centering
\includegraphics[width=\linewidth]{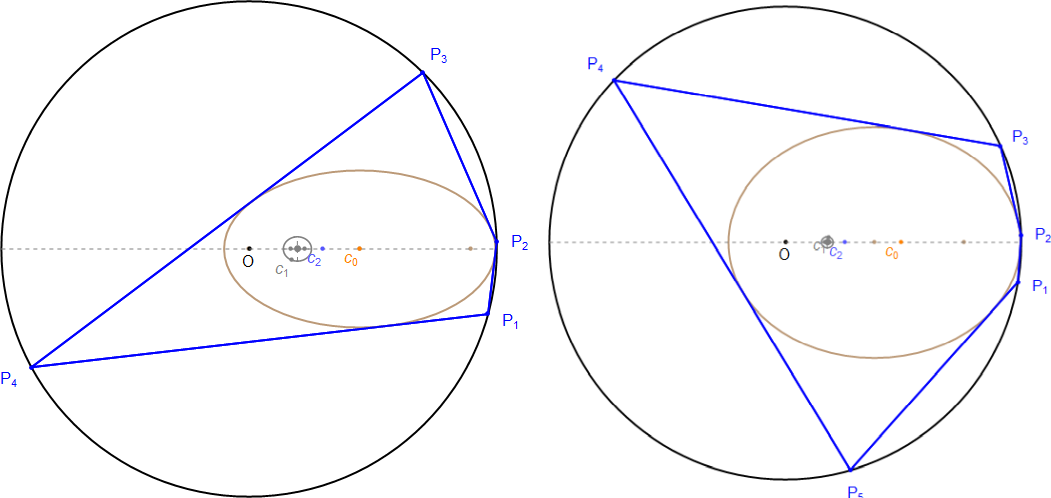}
\caption{MacBeath-like Poncelet families of quadrilaterals (left) and pentagons (right). In both cases The vertex and area centroids $C_0,C_2$ are stationary on the caustic's major axis, whereas the perimeter centroid $C_1$ sweeps an ellipse. In the former case, $C_0$ is located at the center of the caustic. Video: \hrefs{https://youtu.be/e-iOAAFU1H4}}
\label{fig:macbeath-n45}
\end{figure}

\section{\torp{Focal-$X_4$}{Focal-X(4)}}

Referring to \cref{fig:four-families} (bottom left):

\begin{proposition}
The orthocenter $X_4$ is stationary at a focus $f=[\pm c,0]$ of $E$ for the circular caustic with center and radius given by:
\[ C_4 =\left[\pm\frac{a^2 c}{2 a^2-c^2},0\right],\,\,\, r_4 = \frac{a(a^2-c^2)}{2 a^2-c^2}\rd \]
\end{proposition}

Referring to \cref{fig:focal-x4}:

\begin{observation}
For the focal-$X_4$ family, the center (resp. other focus) of $\E$ is $X_{7952}$ (resp. $X_{18283}$). $\L_3$ is an ellipse with a focus at $E$'s center. $\L_{20}$ is twice as large, with a focus on $\E$'s distal focus.
\end{observation}

\begin{figure}
\centering
\includegraphics[width=0.7\linewidth]{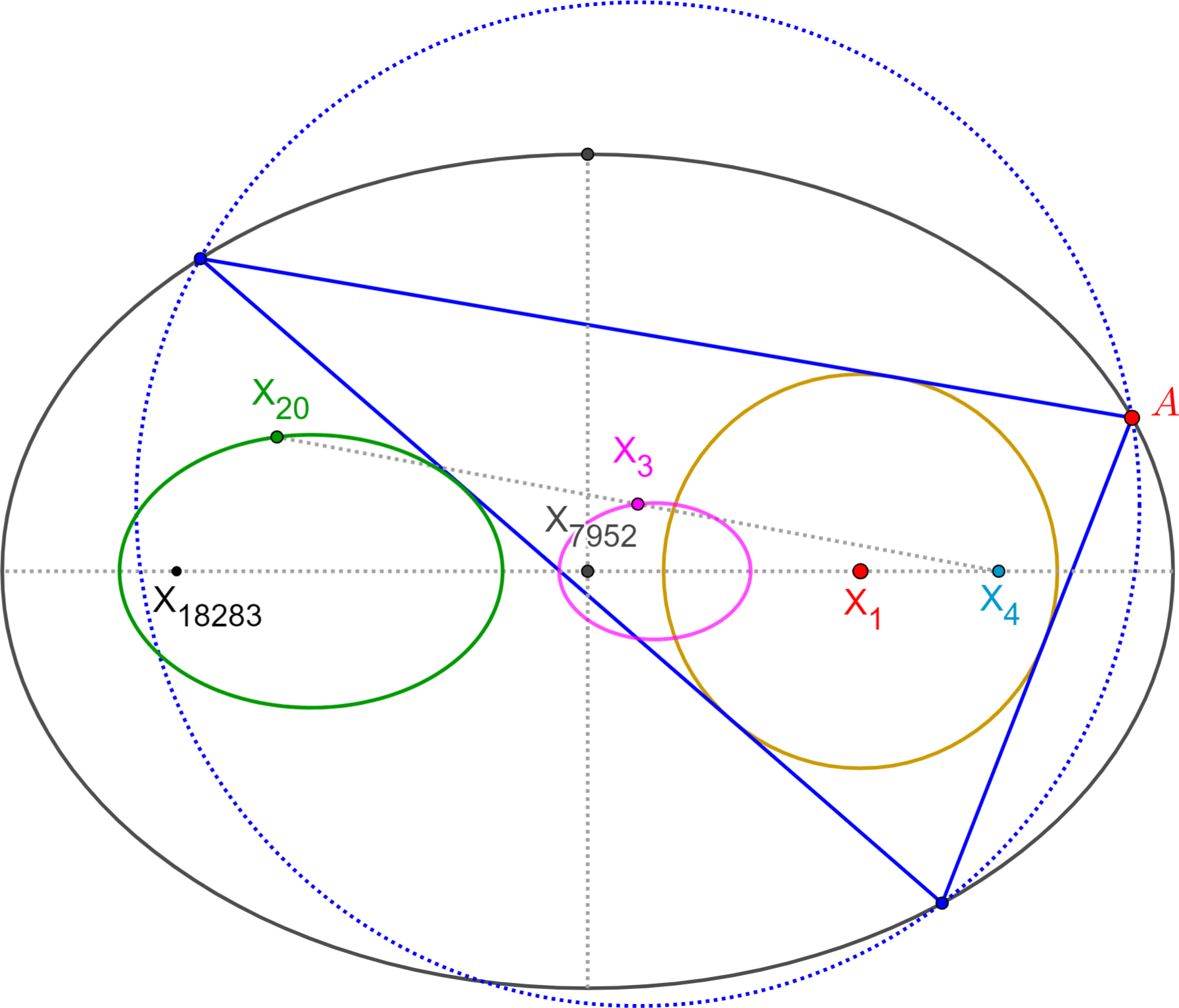}
\caption{The focal-$X_4$ family: the outer ellipse $E$ has foci on the orthocenter $X_4$ and $X_{18283}$ and center on $X_{7952}$. Also shown are the loci of the circumcenter $X_3$, with a focus on the latter center, and that of the de Longchamps point $X_{20}$ (a reflection of the orthocenter $X_4$ on the circumcenter $X_3$), so twice as big, and with a focus on the other focus of $E$. Video: \hrefs{https://youtu.be/x9ce56vvDNI}}
\label{fig:focal-x4}
\end{figure}

\begin{definition}
A triangle's \textit{polar circle} is centered on the orthocenter $X_4$ and has squared radius given by \cite[Polar circle]{mw}:
\[r_{pol}^2=4R^2-\frac{\sum{l_i^2}}{2} \rd\]
This quantity is positive (resp. negative) for obtuse (resp. acute triangles).
\end{definition}

\begin{proposition}
The focal-$X_4$ family conserves the (negative) squared radius of its polar circle at
$r^2_{pol} = -b^4/(a^2 + b^2).$
\end{proposition}

\begin{proof}
The incenter-orthocenter squared distance is given by \cite[Incenter, eqn.6]{mw}:
\begin{equation*}
|X_1-X_4|^2=2r^2+4R^2-\frac{\sum{l_i^2}}{2} = 2r^2+r_{pol}^2\, \cdot
\end{equation*}
Since for this family $|X_1 X_4|$ is fixed as is the inradius $r$. The expression was obtained by CAS-based simplification.
\end{proof}

Let $a,b$ denote the semiaxis' lengths of the MacBeath inconic.

\begin{lemma}
Over the MacBeath family, the sum of squared sidelengths, the sum of double-angle cosines, and the product of cosines are conserved. There are given by:
\begin{align*}
\sum{l_i^2} &= 9R^2-|X_3-X_4|^2 = 
32 a^2 + 4 b^2, \\
\sum{\cos(2\theta_i)} &= \frac{c^2-3a^2}{2a^2}\rc\\
\prod{\cos\theta_i} &=  \frac{\sum{l_i^2}}{8R^2}-1=\frac{b^2}{8a^2}\cdot
\end{align*}
\end{lemma}

\begin{proof}
The first relation is direct from \cite[Circumcenter, eqn.6]{mw}, noting that (i) the circumradius $R$ is fixed and (ii) the caustic foci $X_3,X_4$ are stationary. The second one is obtained via CAS-simplification. The third one is direct from \cite[Circumradius, eqn.5]{mw}.
\end{proof}

Referring to \cref{fig:macbeath-dual} (right): 

\begin{definition}
The \textit{dual} Poncelet family is interscribed between two dual ellipses \cite{walker1950-curves}. They are also homothetic if one is rotated by 90-degrees. 
\end{definition}

\begin{figure}
\centering
\includegraphics[width=.8\linewidth]{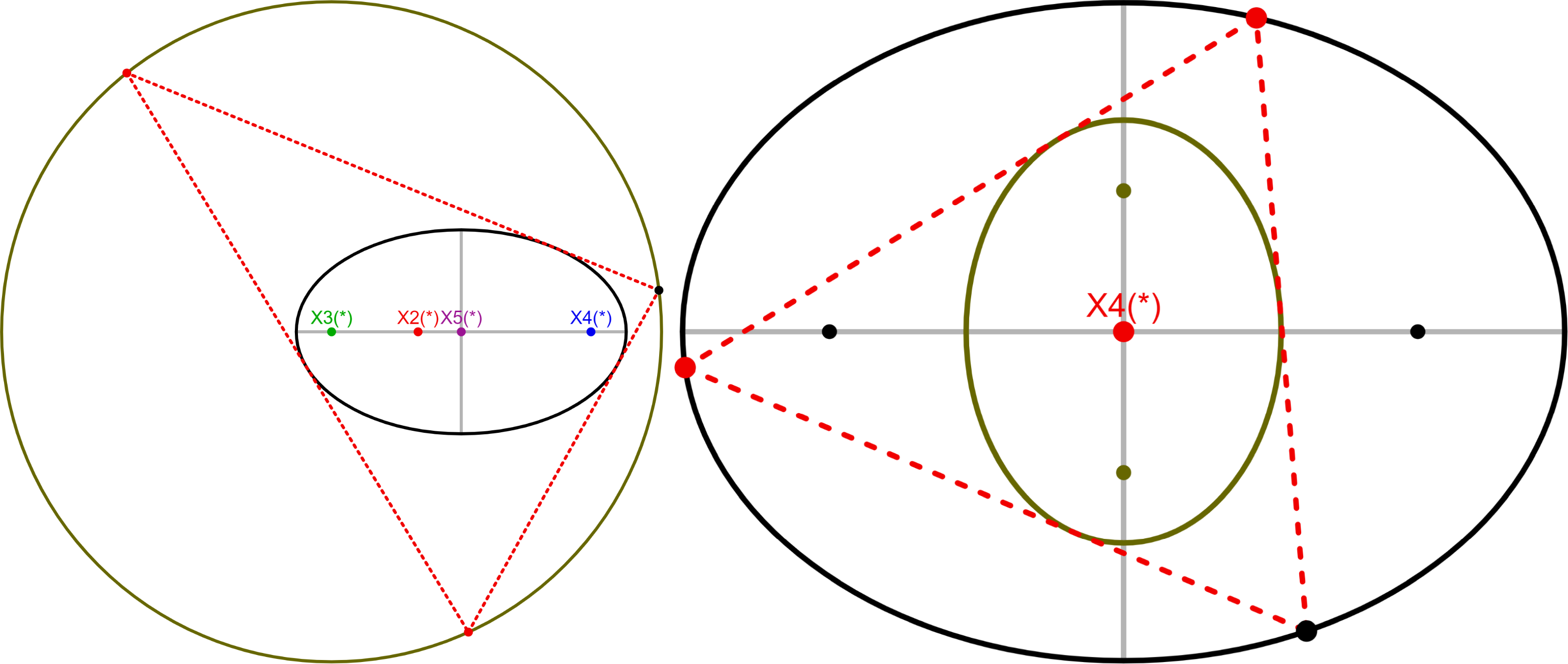}
\caption{\tb{left}: the `MacBeath' family. Live: \hrefs{https://bit.ly/4eVqLAS}; \tb{right}: the `dual' family. Live: \hrefs{https://bit.ly/4eSQSsd}}
\label{fig:macbeath-dual}
\end{figure}

\begin{proposition}
The (i) MacBeath and (ii) dual families conserve $r_{pol}^2<0$. These are given by:
\begin{align*}
r_{pol,macbeath}^2 & = \frac{|X_3-X_4|^2-R^2}{2}=
-2b^2 \rc \\
r_{pol,dual}^2 & = -\frac{a^2 b^2}{a^2 + b^2}\,\cdot 
\end{align*}
\end{proposition}

\begin{proof}
The first is direct from $|X_3-X_4|^2=9R^2-\sum{l_i^2}=2 r_{pol}^2+R^2$ \cite[Orthocenter, eqn.14]{mw}, noting that the MacBeath circumradius $R$ is fixed and the caustic foci on the circumcenter $X_3$ and orthocenter $X_4$ are stationary. For the dual family, the orthocenter $X_4$ is stationary at the common center. Given the first relation it would be sufficient to prove that $|X_3(t)|^2-R^2(t)$ is constant. The final expression is obtained via CAS simplification.
\end{proof}

\begin{observation}
For the dual family, $\L_3$ is an ellipse homothetic to $\E$ with factor
$c^2/( 2(a^2+b^2))$.
\end{observation}

\begin{conjecture}
A Poncelet triangle family maintains the orthocenter $X_4$ stationary only if the conics are configured as focal-$X_4$, MacBeath or Dual.
\end{conjecture}

\section{\torp{Iso-$X_7$}{Iso-X(7)}}

Referring to \cref{fig:four-families} (bottom right):

\begin{proposition}
For the circular caustic with center $C_7=\left[k_7/(2 a),0\right]$ and radius $r_7=b^2/(2 a)$, the Gergonne point will be stationary at:
\[X_7=\left[\frac{2 a k_7}{4 a^2 - b^2},0\right], \]
where $k_7=\sqrt{4 a^4-5 a^2 b^2+b^4}$.
\end{proposition}

\begin{observation}
The iso-$X_7$ family is the polar image of the Brocard porism \cite{reznik2022-brocard-converging,shail1996-brocard} with respect to the circumcircle, i.e., its tangential triangle. The former's stationary Gergonne point $X_7$ coincides with the stationary symmedian point $X_6$ of the Brocard porism \cite{garcia2020-family-ties}.
\end{observation}

\begin{proposition}
The family conserves the sum of half-angle tangents, and this is given by:
\[ \sum_{i=1}^{3}{\tan\frac{\theta_i}{2}} = \frac{\sqrt {4\,{a}^{2}-{b}^{2}}}{a}\rd \]
\end{proposition}

The $N>3$ generalization of the Brocard porism is known as the Harmonic family, studied in \cite{roitman2022-harmonic}. 

\begin{conjecture}
The sum of half-angle tangents is also conserved for the polar image of the Harmonic family, for all $N>3$.
\end{conjecture}

Since the Gergonne center $X_7$ is stationary, the quantity $|X_1 X_7|$ will also be conserved. For a generic triangle with semiperimeter $s$, this is given by \cite[Corollary 4.2]{queiroz2012-gergonne}:
\[ |X_1-X_7|^2 = r^2\left[1-\frac{3 s^2}{(r+4 R)^2}\right]\rd \]

\begin{proposition}
Over Iso-$X_7$ triangles, the previous expression reduces to:
$|X_1-X_7|^2=b^4c^2/\left(4 a^2 (4a^2 - b^2)\right)$.
\end{proposition}

\begin{definition}
The Adams circle of a triangle is centered on the incenter $X_1$ and has radius $R_A$ given by \cite[Adams' Circle]{mw}: $R_A = \frac{r \sqrt{\rho^2 - l_1 l_2 l_3 s - \rho s^2}}{\rho - s^2}$, where $\rho = l_1 l_2+l_2 l_3+l_3 l_1$ and $s$ is the semi-perimeter.
\end{definition}

\begin{proposition}
The Iso-$X_7$ family conserves $R_A$ and its value is given by:
\[ R_A=\frac{b^2}{2a} \sqrt{\frac{5a^2 - b^2}{4a^2 - b^2}}\, \rd \]
\end{proposition}

\bibliographystyle{maa-eprint}
\bibliography{refs,refs_00_book,refs_01_pub,refs_03_sub,refs_04_unsub}

\end{document}